\documentclass[11pt]{amsart}

\usepackage{epsfig}
\usepackage{amsfonts}
\usepackage{amssymb}
\usepackage{amsmath}
\usepackage{amsthm}
\usepackage{amscd}
\usepackage{amstext}
\usepackage{latexsym}

\usepackage{setspace}

\usepackage{enumerate}

\usepackage{xy}
\usepackage{footmisc}
\input xy
\xyoption{all}

\theoremstyle{plain}

\newtheorem{theorem}{Theorem}[section]

\newtheorem{lemma}[theorem]{Lemma}
\newtheorem{corollary}[theorem]{Corollary}
\newtheorem{proposition}[theorem]{Proposition}
\newtheorem{definition}[theorem]{Definition}

\newtheorem*{division problem}{Division Problem}

\theoremstyle{remark}
\newtheorem{remark}[theorem]{Remark}
\newtheorem{example}[theorem]{\textnormal{\textbf{Example}}}

\newcommand{\CC}{\mathbf{C}}
\newcommand{\QQ}{\mathbf{Q}}

\newcommand{\JJ}{\mathcal{J}}
\newcommand{\OO}{\mathcal{O}}

\newcommand{\al}{\alpha}

\newcommand{\ep}{\epsilon}

\newcommand{\qa}{\quad}
\newcommand{\vp}{\varphi}

\newcommand{\lsm}{\preceq}
\newcommand{\gsm}{\succeq}

\newcommand{\noi}{\noindent}

\providecommand{\abs}[1]{\left|#1\right|}

\providecommand{\norm}[1]{\lVert#1\rVert}

\DeclareMathOperator{\spann}{span}

\DeclareMathOperator{\Siu}{Siu}

\begin{document}



\title{Equivalence of plurisubharmonic singularities and Siu-type metrics}


\author{Dano Kim}


\normalsize

\maketitle

\begin{abstract}

\noindent  We show by an example that the (equivalence class of) singularity of a plurisubharmonic function cannot be determined by the data of its Lelong numbers, in a nontrivial sense. Such an example is provided by Siu-type singular hermitian metrics associated to a line bundle. We also show that a Siu-type metric has analytic singularities if and only if the section ring of the line bundle is  finitely generated. 
\end{abstract}

\let\thefootnote\relax\footnotetext{\noindent This work was supported by the  National Research Foundation of Korea grants NRF-2012R1A1A1042764 and No.2011-0030795, funded by the Republic of Korea government.}

\section{Introduction}

  Plurisubharmonic (psh for short) functions appear naturally and have numerous applications in complex analysis and complex algebraic geometry (see \cite{D} for an excellent reference). A psh function often appears as a local weight function of a singular hermitian metric of a line bundle on a complex manifold.  
  
  We say that a psh function $\vp$ has a \textit{singularity} at $p$ if $\vp (p) = -\infty$.   Two of the most important ways to measure the singularity are in terms of Lelong numbers and in terms of multiplier ideal sheaves~\cite{D},\cite{N}. It is well known that general plurisubharmonic functions with non-analytic singularities (see Definition~\ref{analy}) may have subtle and complicated behaviour. 
 
 There are the following conditions (A), (B) and (C) which compare singularity of two psh functions $u$ and $v$ defined on a complex manifold.  The  fundamental result of Boucksom, Favre and Jonsson \cite[Theorem A]{BFJ}, combined with the recent `strong openness theorem' of Guan and Zhou \cite{GZ} (see also \cite{H} and Remark~\ref{openn}), shows the implications (A) $\to$ (B) $\to$ (C) : 
 
\begin{enumerate}[(A)]

\item

  $u$ and $v$ have equivalent singularities. 
\\
 
\item  (B-1)  For every real number $m > 0$, the equality of multiplier ideal sheaves $\JJ(m u) = \JJ(m v)$ holds. 

 \noindent   (B-2) $u$ and $v$ have the same Lelong number at every point of all proper modifications over $X$. 

\item   $u$ and $v$ have the same Lelong numbers at every point of $X$. 

\end{enumerate}

It is the highly nontrivial equivalence of (B-1) and (B-2) that the above theorems imply.    Note that (B-1) implies (C) by Demailly approximation~\cite{D}. All three conditions are inequivalent: first of all, there is a local trivial example showing that (A) and (B) are inequivalent : 
when $u-v$ is a psh function $\vp$ with zero Lelong number at every point but not locally bounded (e.g. $u-v = -\sqrt{-\log \abs{z}}$).  (For (B) and (C), see Example~\ref{jonsson}.) Note that this trivial example cannot be globalized in that it cannot be used to give two singular hermitian metrics $e^{-u}$ and $e^{-v}$ of one line bundle for which (B) holds but (A) does not. 


  In this paper, we give a nontrivial global  example showing the inequivalence of (A) and (B).
  Such an example should involve a psh function which does not have analytic singularities: if we assume that both $u$ and $v$ have analytic singularities, then the conditions (A) and (B) are equivalent for them as it is easy to see from their common log-resolution. 
  

 Such a psh function with non-analytic singularities is provided by  Siu-type singular hermitian metric $h_{\Siu}$ of a $\QQ$-effective line bundle $L$ which is analogous to the stable base locus of $L$ in algebraic geometry. It was defined and used successfully in Siu's proof of invariance of plurigenera~\cite{S98} in the case of general type. 
 The analytic methods of \cite{S98} were adapted to one in algebraic language (\cite{Ka}, see also \cite[Section 11.5]{L}), replacing the use of $h_{\Siu}$ by the asymptotic multiplier ideal sheaf $\JJ (  \norm{L} )$.  
  
  The definition of $h_{\Siu}$ actually involves a choice of coefficients and it is natural to ask whether the metric is nevertheless uniquely determined  up to equivalence of singularities. We show that the uniqueness fails whenever the section ring of the line bundle $L$ is not finitely generated (\textbf{Theorem~\ref{siusiu}}). In fact, in such a case, there exists an infinite family of Siu-type metrics $h_{\Siu,i} (i \in I)$ such that none of them have equivalent singularties (\textbf{Corollary~\ref{siuco}}).

 As for their multiplier ideal sheaves, \cite[Theorem 1.11]{DEL} implies that for a big line bundle $L$, we have the inclusions (for every real $m > 0$) 
 
\begin{equation}\label{deldel}
 \JJ_{+} (h_{\min}^m) \subset  \JJ ( m \norm{L} ) \subset \JJ(h^m_{\Siu, i}) \subset \JJ (h^m_{\min}) \quad \quad \forall i \in I
 \end{equation}
 
\noi where $h_{\min}$ is the singular hermitian metric of minimal singularities associated to a pseudo-effective line bundle $L$, due to Demailly. Note that \cite[Theorem 1.11]{DEL} assumes projectivity since Ohsawa-Takegoshi theorem is used. Also the proof of \cite[Theorem 1.11]{DEL} therein for $m=1$ works for every real $m > 0$ to yield \eqref{deldel} as we see in the last lines of the proof.

  Now the inclusions in \eqref{deldel} are all equalities due to the openness theorem in \cite{GZ}.  In particular, it follows that $\JJ(h^m_{\Siu, i}) = \JJ(h^m_{\Siu, j})$ for all $i,j \in I$ and all real $m > 0$. Thus we obtain the following theorem as a consequence of Theorem~\ref{siusiu} (denoting $e^{-\vp_i} := h_{\Siu, i}$). 

\begin{theorem}\label{maint}
 
 Let $L$ be a big line bundle on a  smooth projective variety $X$ such that the section ring of $L$ is not finitely generated. Then there exists an infinite family $\{ e^{-\vp_i} \}_{i \in I}$ of singular hermitian metrics of $L$ such that for all $j \neq k$ in $I$, the equality of multiplier ideal sheaves $\JJ(m \vp_j) = \JJ(m \vp_k)$ holds for every real $m > 0$ and yet $\vp_j$ and $\vp_k$ do not have equivalent singularities. 

\end{theorem}

 In the last section of this paper, we also show that a Siu-type metric has analytic singularities if and only if the section ring of the line bundle is finitely generated (\textbf{Theorem~\ref{thm2}, Corollary~\ref{fgs}}). This strengthens the previously known result that the section ring of the line bundle is finitely generated if and only if a Siu-type metric is equivalent to a finite partial sum (see Proposition~\ref{partial}). Corollary~\ref{fgs} is proved again following the main theme of this paper: comparison of the relations (A) and (B) for psh functions. In particular, Corollary~\ref{fgs} gives us many new examples of psh functions with non-analytic singularities, resulting from each instance of a  line bundle with non-finitely generated section ring on a compact complex manifold. 
 
 We remark that Section 3 of this paper had existed as a note `A remark on Siu-type metric' (e.g. in References of~\cite{M1}). 
 
\begin{remark}\label{openn}

  In this paper, the strong openness theorem $\JJ_+ (\vp) = \JJ (\vp)$ for a psh function $\vp$ of \cite{GZ} is used in three different ways: the equivalence of (B-1) and (B-2), Theorem~\ref{maint} (as above) and Proposition~\ref{zero}. Note that Theorem~\ref{thm2} also uses the equivalence of (B-1) and (B-2).   As we mentioned, the strong openness conjecture $\JJ_+ (\vp) = \JJ (\vp)$ was recently proved by \cite{GZ} in all dimensions, after which \cite{H} presented a simplified proof. Before \cite{GZ}, it was proved in dimension $2$ case by \cite{FJ1}, \cite{FJ2}, in the case of toric psh functions by \cite{G} and in the case when $\JJ(\vp) = \OO_X$ by \cite{B}. Also \cite{JM},  \cite{GZ1}, \cite{Le} are related to the recent activity on this topic.

\end{remark}

 \qa

\noi \textit{Acknowledgement.}  The author wishes to thank S{\'e}bastien Boucksom for his crucial suggestion of using an argument of \cite{BEGZ} toward Theorem~\ref{siusiu} and Mihai P\u aun for showing him the proof of Proposition~\ref{zero} in 2009.  He also thanks Mattias Jonsson for telling him Example~\ref{jonsson}, Nikolay Shcherbina for raising a question on using multiplier ideal sheaves to distinguish psh functions and the anonymous referee for helpful comments.

\section{  PSH with zero Lelong numbers  }

 We begin with the following slightly generalized definition of a psh function with analytic singularities (cf. \cite[Definition 1.10]{D}). 

\begin{definition}\label{analy}

 We say that a psh function $\vp$ on a complex manifold \textbf{has analytic singularities} if it can be locally written as $\vp = c \log (\sum_{i=1}^N \abs{f_i}^{2} ) + v $ where $f_i$ are local holomorphic functions, $c > 0$ a real number and $v$ a locally bounded function. 

\end{definition}

 If $\vp$ does not satisfy this definition, we will say that $\vp$ has non-analytic singularities.    Perhaps the simplest examples of a nontrivial (i.e. with non-analytic singularities) psh function with zero Lelong number everywhere are given as follows. (For another example, see Example~\ref{siuzero}.)
 
\begin{example}

 Let $\vp_\alpha (z) = - (- \log \abs{z})^{\alpha} $ for $0 < \alpha < 1$ on $B := \{ \abs{z} < 1 \} \subset \CC^n$. Let $\vp_0 (z) = - \log (- \log \abs{z}) $ on $B$. They are psh and the Lelong number $\nu (\vp_\alpha, p) = 0$ for every $p \in B$ ($0 \le \alpha < 1$). 

\end{example}

 The following proposition generalizes \cite[Proposition 1.1]{N} which was the trivial case of $\psi$ smooth. 

\begin{proposition}[M. P\u aun]\label{zero}

 Let $\vp$ and $\psi$ be two plurisubharmonic functions on a complex manifold $X$. Suppose that $\psi$ has zero
 Lelong number at every point of $X$. Then we have  $ \JJ( m(\vp + \psi) ) = \JJ (m \vp ) $ for every real $m > 0$. 

\end{proposition}

\begin{proof}
 We may look at an open bounded domain $\Omega \subset X$ and may assume that $m=1$.  Suppose that $f$ is holomorphic on $\Omega$ and $f \in
 \JJ(\vp)$. It suffices to show that $f \in \JJ(\vp + \psi)$. From the openness theorem \cite{GZ}, we have $\JJ_+ (\vp) = \JJ
 (\vp)$. Hence there exists $\ep > 0$ such that $f \in \JJ((1+\ep) \vp)$. Let $\frac{1}{p} := \frac{1}{1+\ep}$ and
 $\frac{1}{q} := \frac{\ep}{\ep + 1}$. From the H\"older inequality, 

 $$ \int_\Omega \abs{f}^2 e^{-\vp} e^{-\psi} dV \le  \left(  \int_\Omega  ( \abs{f}^{\frac{2}{p}} e^{-\vp} )^p
 dV  \right)^{\frac{1}{p}}
  \left( \int_\Omega  ( \abs{f}^{\frac{2}{q}}  e^{-\psi} )^q   dV   \right)^{\frac{1}{q}}   .$$ Here the RHS is finite since $f \in
  \JJ( (1+\ep) \vp)$ and the Lelong number of $\psi$ is zero at every point, using Skoda's lemma \cite{D} (5.6). 

\end{proof}

\noindent  We have immediate corollaries. Note that (2) is a special case of (1). 

\begin{corollary}

\begin{enumerate}

\item
 Conditions (B) and (C) (in the first page) are equivalent for $u$ and $v$ if $u-v$ is psh. 

\item
 
  If a psh function $u$ on $X$ has zero Lelong number at every point of $X$, then it has zero Lelong number at every point of all proper modifications over $X$. 
 
\end{enumerate}

\end{corollary}

 The following example  shows that the conditions (B) and (C) are indeed inequivalent, in general. 
 
\begin{example}(Jonsson \cite{J})\label{jonsson}

  On $\CC^2$ with coordinates $(x,y)$, let $u = \log \max \{ \abs{x}, \abs{y} \}$, i.e. the psh function associated to the ideal generated by $x$ and $y$. Also let $v = \log \max \{ \abs{x}^2, \abs{y} \}$. Then $u$ and $v$ have the same Lelong numbers everywhere on $\CC^2$: $1$ at the origin and $0$ elsewhere. Now let $\pi : X \to \CC^2$ be the blowup of the origin and let $E \subset X$ be the exceptional divisor. Then $\pi^* u$ has Lelong number $1$ at every point of $E$ while $\pi^* v$ has Lelong number $2$ at one point of $E$.

\end{example}

 If we also consider functions that are given as the difference of two psh functions, we may define their Lelong number at a point in the following obvious way: $\nu(\psi_1 - \psi_2, x) := \nu(\psi_1, x) - \nu(\psi_2, x)$.  In this generality, the conclusion of Proposition~\ref{zero} is not true any more, for example when we take $\psi = - 2 \vp_0$ from the above example. Then $\JJ ( \vp - 2 \vp_0)$ is the analytic adjoint ideal sheaf defined in \cite{G} and it is different from $\JJ(\vp)$ in general.  It also follows that the above Skoda's lemma does not hold in this generality of difference of two psh functions.

\section{Singular hermitian metrics of Siu-type}

 Toward defining a Siu-type metric in this section, we first revisit the basic properties of singular hermitian metrics defined by sections of a line bundle in a rather careful manner, since that was how we came to discover the subtlety of the definition of a Siu-type metric.  Let $L$ be a line bundle and $h_1= e^{-\varphi_1}$ and $h_2= e^{-\varphi_2}$ two singular hermitian metrics. Following the usual convention, we often use $\vp_1$ to refer to the metric $h_1$.  Also following \cite[Definition 0.5]{D14}, we say $h_1 = e^{-\varphi_1}$ is \textbf{less singular} than $h_2 = e^{-\varphi_2}$ and write $\vp_1 \lsm \vp_2$ and $h_1 \lsm h_2$ if the local weight functions satisfy $\vp_2 \le \vp_1 + O(1)$. We say $\vp_1$ and $\vp_2$ have \textbf{equivalent singularities} and write $\vp_1 \sim \vp_2$ if $\vp_1 \lsm \vp_2$ and $\vp_2 \lsm \vp_1$.

 Let $g_1, \cdots, g_k \in H^0(X, L)$ be holomorphic sections and $V$ the subspace spanned by them. These sections define a singular metric $e^{-\vp}$ where $\vp = \log (\abs{g_1}^2 + \cdots + \abs{g_k}^2)$ (the equality here understood as usual, either by local frames or by using an auxiliary smooth metric of $L$, see \cite[Example 3.14]{D}). 

\begin{proposition}\label{equiv}

 Let $W \subset V \subset H^0(X, L)$ be subspaces.

\begin{enumerate}

\item

 Suppose that $W = \spann \{w_1, \cdots, w_k \}$ and $V = \spann \{ v_1, \cdots, v_l \}$. Then we have  $$\log (\abs{w_1}^2 + \cdots + \abs{w_k}^2)    \gsm \log(\abs{v_1}^2 + \cdots + \abs{v_l}^2) .$$

\item

 The metric defined by the choice of a spanning set of $V$ as in (1) is unique up to equivalence.

\end{enumerate}

\end{proposition}

\begin{proof}

 (2) follows from (1) by taking $W = V$ and letting $\{ v_i \}$ and $\{ w_j \}$ be two arbitrary spanning sets. For (1), first note that $k \le l$ since $W \subset V$. For each $i$, one can write $w_i = \sum a_{ij} v_j$. Then Cauchy-Schwarz gives $\abs{w_i}^2 \le (\sum_j \abs{a_{ij}}^2)  (\sum_j \abs{v_j}^2) $, hence (1).

\end{proof}

 Now let $L$ be a line bundle on a compact complex manifold and suppose that $L$ is \textit{$\QQ$-effective}, i.e. for some $k \ge 1$, the tensor power $kL$ has some nonzero global holomorphic section. For such $L$, holomorphic sections $g_1, \cdots, g_p$ of a multiple $kL$ define a singular hermitian metric for $L$ in the following two ways: $ \vp_1 =  \frac{1}{k} \log (\sum \abs{g_i}^2) $ and $ \vp_2 =  \log (\sum \abs{g_i}^{\frac{2}{k}}) $. These two have equivalent singularities since we have \cite{F05} the following elementary inequalities from H\"{o}lder:

\begin{equation}\label{fujino}
 \frac{1}{\sqrt[k]{p}}  (\sum_{i=1}^p \abs{g_i}^2)^{\frac{1}{k}} \le  \sum_{i=1}^p \abs{g_i}^{\frac{2}{k}}   \le \frac{p}{\sqrt[k]{p}} (\sum_{i=1}^p \abs{g_i}^2)^{\frac{1}{k}}.
\end{equation}

\noi In view of this and Proposition~\ref{equiv}, we propose the following convenient 

\noi \textbf{Notation.} Given a subspace
 $V \subset H^0(X, L)$, the metric defined by a basis $\{g_1, \cdots, g_k \}$ of $V$ and its equivalence class is denoted by $\log \abs{V}^2$. Similarly the metric for $L$ defined by a basis of a subspace $V \subset H^0(X, kL)$ and its equivalence class is denoted by $\log \abs{V}^{\frac{2}{k}}$.
 
 One drawback of this notation is that when we define a Siu-type metric below using an infinite sum, the ambiguity of equivalence class cannot be allowed because of convergence.  So in that case, we understand the notation $\log \abs{V}^{\frac{2}{k}}$ as referring to a specific choice of a basis of $V$ as in  $$\log \abs{V}^{\frac{2}{k}} =\log \left( \sum_{i=1}^p \abs{g_i}^{\frac{2}{k}} \right)    .$$

\noi Now let $R(L)$ be the section ring of a line bundle $L$ and  let $S$ be a graded linear system~\cite[Definition 2.4.1]{L} of $L$ : 

  $$S := \oplus_{k \ge 0} S_k \subset R(L) := \oplus_{k \ge 0} H^0 (X, kL)  $$ on a compact complex manifold.  One can then define a metric for the line bundle $L$ by considering all graded pieces $S_k  (k \ge 1)$ at the same time. Let $(L, g)$ be an auxiliary choice of a smooth hermitian metric without curvature condition. Let $C_k := \sup_X   \abs{S_k}_g^{\frac{2}{k}}  $ for each $k \ge 1$. We say a sequence of positive real numbers $( \ep_k )_{k \ge 1}$ is \textit{admissible} if $\sum_{k \ge 1} \ep_k C_k $ converges. Clearly, this property is independent of the choice of $g$. 

 \begin{definition}[Siu~\cite{S98}]\label{series}
 
 Let $S \subset R(L)$ be a graded linear system of a line bundle $L$ on a compact complex manifold as above. With an admissible choice of a sequence  of positive numbers $( \ep_k )_{k \ge 1}$,  the infinite series

 $$\rho_S = \log \left( \sum_{k \ge 1}  \ep_k  \abs{S_k}^{\frac{2}{k}} \right) $$ gives a singular hermitian metric $(L,e^{-\rho_S})$ whose local weight functions are psh, and we call $e^{-\rho_S}$ a \textbf{Siu-type metric} of $S$ for $L$.

\end{definition}

\noi  In terms of local trivializations of $L$, each local weight function in the place of $\rho_S$  is psh since it is the upper envelope of the uniformly bounded above family of psh functions given by
the partial sums (converging uniformly over compact subsets).

\begin{example}\label{siuzero}

 In fact, in Definition~\ref{series}, we do not really need each subspace $S_k$ of $H^0 (X, kL)$ to form a graded linear system in order for $\rho_S$ to make sense. Also, as far as the vector spaces $S_k$ are finite dimensional, the definition works in a local setting as well. For example, for an admissible choice of coefficients, $\vp = \log ( \sum_{k \ge 1}  \ep_k  \abs{z_1}^{\frac{2}{k}} )$ gives a psh function in a neighborhood of the origin in $\CC^n$ (with coordinates $(z_1, \cdots, z_n)$) whose Lelong numbers are zero everywhere.  

\end{example}

 We note that, a priori, different choices of  $( \ep_k )_{k \ge 1}$  could yield different Siu-type metrics even up to equivalence of singularities. The main theorem  of this section says that this is indeed the case,  precisely when the section ring is not finitely generated.
  On the other hand  in the case of finite generation, Siu-type metric is uniquely determined up to equivalence of singularities, independently of the choice of  $( \ep_k )_{k \ge 1}$.  This is well known in the following result, at least in the case of a full graded linear system, i.e. $S = R(L)$.  

\begin{proposition}\label{partial} \cite[Lemma 6.6]{BEGZ} cf.  \cite[Theorem 3.7]{S06}

 The section ring $S = R(L)$ of a $\QQ$-effective line bundle $L$ is finitely generated if and only if every Siu-type metric $\rho_S = \log \left( \sum_{k \ge 1}  \ep_k  \abs{S_k}^{\frac{2}{k}} \right)$  of $L$ is equivalent to a finite sum  $\log \left( \sum^{k_0}_{k \ge 1}  \delta_k  \abs{S_k}^{\frac{2}{k}} \right)$ for some $k_0 \ge 1$ and coefficients $\delta_k > 0$. 

\end{proposition}

 We remark that the method of \cite[Theorem 3.7]{S06} (using Skoda division theorem) for a canonical line bundle can be generalized to a general line bundle $L$ to give a proof of this proposition.

 Now we show that the definition of a Siu-type metric indeed depends on the choice of the coefficients $( \ep_k )_{k \ge 1}$ in general. 
Thus we cannot refer to it as \textit{the} Siu-type metric of $L$ even up to equivalence of singularities. 

\begin{theorem}\label{siusiu}

 Let $L$ be a $\QQ$-effective line bundle such that the section ring $R(L) =: S$ is not finitely generated.  If $e^{\rho_1} = \sum_{k \ge 1}  \ep_k  \abs{S_k}^{\frac{2}{k}} $ is a given Siu-type metric of $L$, there always exists another Siu-type metric $e^{\rho_2 }= \sum_{k \ge 1}  \al_k  \abs{S_k}^{\frac{2}{k}} $ such that $\rho_2$ is more singular than $\rho_1$ and strictly so, i.e. $\rho_2$ and $\rho_1$ do not have equivalent singularities.

\end{theorem}

\begin{proof}

We follow the argument of the proof for  \cite[Proposition 6.5]{BEGZ}, Proposition 6.5 replacing a metric with minimal singularities $\psi$ therein by a Siu-type metric.

 For convenience, let us denote $ e^{\phi_k} := \abs{S_k}^{\frac{2}{k}}$ and thus $e^{\rho_1} = \sum_{k \ge 1}  \ep_k  e^{\phi_k}$. Given $\sum_{k
 \ge 1} \ep_k $, choose another admissible convergent series of positive numbers $\sum_{k \ge 1} \alpha_k$ such that $\al_k < \ep_k$ and  the series $\sum_{k \ge 1}
 \frac{\al_k}{\ep_k}$ is convergent. Let $e^{\rho_2} = \sum_{k \ge 1}  \al_k  e^{\phi_k}$. Clearly  $\rho_2 \gsm \rho_1$.  We will show that $\rho_2$ is strictly more singular than $\rho_1$.

  Since $R(L)$ is not
 finitely generated, Proposition~\ref{partial} says that none of the partial sum metric $\sum^k_{i= 1}  \al_i  e^{\phi_i}$ can have equivalent singularities with the Siu-type metric $e^{\rho_1}$ for every $k \ge 1$. Thus the function defined by the quotient of metrics $\left( \sum^k_{i= 1} \al_i  e^{\phi_i} \right)  / e^{\rho_1} $ is not bounded away from zero (while bounded from above).

 Therefore, for
 any given increasing sequence of positive numbers $(C_k)_{k \ge 1}$, we can find a sequence of points in $X$,
 $(x_k)_{k \ge 1}$ such that when evaluating at the point $x_k$, we have the inequality
 
  $$ \left( \sum^k_{i= 1} \al_i  e^{\phi_i (x_k)} \right)   \le e^{\rho_1 (x_k)}  e^{ - C_k}.$$   
  
\noi Again note that this and the next inequalities make sense when we consider the quotient of hermitian metrics on both sides as a function.  It follows that (for every $k \ge 1$)
 
\begin{align*}
  e^{\rho_2 (x_k) } \le e^{\rho_1 (x_k)}  e^{ - C_k} + \sum_{i  > k} {\al_i} e^{\phi_i (x_k)}   \le e^{\rho_1 (x_k)}  e^{ - C_k} + \sum_{i  > k} \frac{\al_i}{\ep_i}  e^{\rho_1 (x_k)}
\end{align*}

\noi since $\ep_i e^{\phi_i} \le  \sum_{i \ge 1} \ep_i e^{\phi_i}  = e^{\rho_1} $. Now let

 $$ \delta_k := \sum_{i > k} \frac{\al_i}{\ep_i} .$$

\noi  If we choose $C_k \to \infty$ fast enough to ensure that  $e^{- C_k} \le \delta_k$,
then we
 have   $  \rho_2 (x_k) - \rho_1 (x_k) \le \log 2 \delta_k $. Since $\delta_k \to 0$ as $k \to \infty$,  the two singular hermitian metrics $\rho_1$ and $\rho_2$ cannot have equivalent
 singularities.

\end{proof}

\begin{corollary}\label{siuco}

If $L$ is a $\QQ$-effective line bundle such that the section ring $R(L)$ is not finitely generated, then there exists an infinite family of Siu-type metrics for $L$ such that none of them have equivalent singularties.

\end{corollary}

\section{Finite generation and analytic singularities}  

 A line bundle $L$ on a compact complex manifold is pseudoeffective if and only if $L$ carries at least one singular hermitian metric with nonnegative curvature current. For a pseudoeffective line bundle, the \textit{singular metric with minimal singularities} $h_{\min} = e^{-\vp_D}$ always exists~\cite{D}, unique up to equivalence of singularities. When $L$ is moreover $\QQ$-effective, a Siu-type metric $\vp_S := \rho_S$ is defined and satisfies $\vp_D \lsm \vp_S$. If the section ring of $L$ is finitely generated, then \cite{BEGZ} shows that $\vp_D \sim \vp_S$ and they have analytic singularities. 
  
  In this section, we study the case when the section ring of $L$ is not finitely generated. In this case, from Theorem~\ref{siusiu} and \cite{BEGZ}, we get a sequence of inequivalent (classes) of singular hermitian metrics: 
  
  $$ \vp_D \lsm \vp_{S,1} \lsm \vp_{S, 2} \lsm \cdots .$$ (Note that $\vp_{S, i}$ does not necessarily exhaust all possible Siu-type metrics of $L$.) Because of the inequivalence, at most one of them can have analytic singularities. Using Theorem~\ref{thm2}, we will show that only $\vp_D$ can possibly have analytic singularities. In fact, for $\vp_D$, both of the two cases can happen. 
 
\begin{example}
(A line bundle for which the section ring is not finitely generated and the metric $\vp_D$ \textit{does not have analytic singularities}.)  
 See \cite[Example 5.4]{BEGZ} (and \cite[Example 5.2]{F13}) for such a (nef and big) line bundle $L$ whose $\vp_D$ has zero Lelong numbers everywhere but not locally bounded. In particular, $L$ is not semipositive. A similar example was given in  \cite{K05}, \cite[Example 2.14]{K07} where only non-semipositivity was checked.

\end{example}

\begin{example}
(A line bundle for which the section ring is not finitely generated and the metric $\vp_D$ \textit{does have analytic singularities}.)   \cite[Example 5.4, Example 5.8]{F13} each gives a semipositive line bundle $L$ whose section ring is not finitely generated. 

\end{example}

 We also note that it is shown in \cite{M} that if the minimal metric $\vp_D$ has analytic singularities, then the line bundle $L$ admits a birational Zariski decomposition. 

 We now turn to Siu-type metrics. The first example of a non-finitely generated section ring $R(L)$ due to Zariski \cite[2.3.A]{L} is worth mentioning: the line bundle $L$ on $X$ of dimension $2$ has the property that there exists a smooth curve $C \subset X$ such that the linear system $| mL - C|$ is free for every $m \ge 1$. Therefore a Siu-type metric $\vp_S$ locally looks like Example~\ref{siuzero}. In particular, the polar set of $\vp_S$ is $C$ but $\vp_S$ has zero Lelong number everywhere. So $\vp_S$ does not have analytic singularities since it is not locally bounded.
 
 In general,  we will show that whenever $L$ is a line bundle whose section ring is not finitely generated, a Siu-type metric has non-analytic singularities.  In particular, we obtain many new examples of psh functions with non-analytic singularities. We need the following theorem. 
 
 \begin{theorem}\label{thm2}

 Let  $u$ and $v$ be two psh functions (or singular hermitian metrics) on a complex manifold $X$ such that   they have the same Lelong numbers on every proper modification over $X$. If $u$ has analytic singularities, then $u$ is {less singular} than  $v$. 

\end{theorem}

\begin{corollary}\label{fgs}

 Let $X$ be a smooth complex projective variety  and let $L$ be a big line bundle on $X$.  The following are equivalent. 
 
 \begin{enumerate}[(a)]

  \item The section ring of $L$ is finitely generated.

 \item Every Siu-type singular hermitian metric of $L$ has analytic singularities.

 \item At least one Siu-type singular hermitian metric of $L$ has analytic singularities.   
 \end{enumerate}
 
\end{corollary}

 We note that this is much stronger than Proposition~\ref{partial}.

 \begin{proof}[Proof of Corollary~\ref{fgs}]
  
  It is clear that (a) $\to$ (b) $\to$ (c). Suppose that $L$ is not finitely generated. It suffices to show that any Siu-type metric $\vp_S$ of $L$ has non-analytic singularities.   Suppose to the contrary that $\vp_S$ has analytic singularities. As in Introduction, \cite[Theorem 1.1]{DEL} implies that $\vp_S$ and  $\vp_D$ have all the same multiplier ideal sheaves and thus Theorem~\ref{thm2} is applicable.  It follows that $\vp_S \lsm \vp_D$ where $\vp_D$ is the  metric of minimal singularities of $L$. On the other hand, it is always true that $\vp_D \lsm \vp_S$.  Thus $\vp_S \sim \vp_D$ and $L$ is finitely generated by \cite[Proposition 6.5]{BEGZ}, which is contradiction.

 \end{proof}

\begin{proof}[Proof of Theorem~\ref{thm2}]

 Let $w = \sup (u, v)$, the pointwise supremum function which is psh. Then $w$ is less singular than $u$ (i.e. $w \lsm u$). Also $w$ and $u$ have the same Lelong numbers on all proper modifications over $X$ as we see from  Lemma~\ref{same}.

 Let $\pi: X' \to X$ be a log-resolution of $u$ with analytic singularities. Then we get 
 
 \begin{equation}\label{aa}
 \pi^* w \lsm \pi^*u
 \end{equation}
 from $w \lsm u$.   On the other hand, the Demailly approximation sequence $\psi_m$ ($m \ge 1$) of  $\pi^* w$ is determined by the sequence of multiplier ideal sheaves $\JJ (m \pi^* w ) = \JJ(m \pi^*  u)$. Then in fact the sequence $\psi_m$ of psh functions is constant (up to equivalence of singularities) and equal to $\pi^*  u$ since $\JJ(m \pi^*  u)$ is computed by the corresponding SNC divisor on $X'$. By the basic property of Demailly approximation~\cite[Theorem 13.2 (a)]{D}, we get
\begin{equation}\label{bb}
  \pi^* w \gsm \psi_m \sim \pi^* u.
\end{equation} 
  Therefore, it follows from \eqref{aa},\eqref{bb} that $\pi^* w \sim \pi^*  u$ and so $u \sim w \lsm v$.  The theorem is proved. 
 
\end{proof}

\begin{lemma}\label{same}

Let $u$ and $v$ be psh functions on a complex manifold $X$ and let $a := \nu (u,x)$ and $b:= \nu(v, x)$ be their Lelong numbers at a point $x \in X$. For $w := \sup (u,v)$ the pointwise supremum psh function, we have $\nu (w, x) = \min (a,b)$. 

\end{lemma}

\begin{proof}

It  follows immediately from the characterization of a Lelong number $\displaystyle \nu(u,x) = \liminf_{z \to x} \frac{u(z)}{\log \abs{z - x} } $. 
\end{proof}

 Finally we remark that it is natural to ask what happens to a Siu-type metric $\rho_S$ in a more general case than Corollary~\ref{fgs} : when $L$ is not big or when $S \subset R(L)$ is not a full graded linear system.

\footnotesize

\bibliographystyle{amsplain}

\qa

\qa

\normalsize

\noi \textsc{Dano Kim}

\noi Department of Mathematical Sciences, Seoul National University

\noi 1 Kwanak-ro, Kwanak-gu, Seoul, Korea 151-747

\noi Email address: kimdano@snu.ac.kr

\noi

%

%

\end{document}